\documentclass{amsart}
\usepackage{amsfonts,amssymb,amsmath,amsthm}
\usepackage{url}
\usepackage{enumerate}

\newtheorem{thrm}{Theorem}[section]
\newtheorem{lem}[thrm]{Lemma}

\theoremstyle{definition}

\numberwithin{equation}{section}

\newcommand{\End}{\operatorname{End}}

\newcommand{\Hom}{\operatorname{Hom}}

\newcommand{\Ann}{\operatorname{Ann}}

\newcommand{\fm}{\frak{m}}

\newcommand{\fn}{\frak{n}}

\author{ Majid Eghbali}

\address{School of Mathematics, Institute for Research in Fundamental Sciences (IPM), P. O. Box: 19395-5746,
Tehran-Iran.}
\email{m.eghbali@yahoo.com}


\begin{document}

\title[ Annihilator of local cohomology modules]{A note on the Annihilator of local cohomology modules in characteristic zero}

\thanks{The author is supported in part by a grant from IPM}.

\begin{abstract}
We give an alternative proof to the annihilator of local cohomology
in characteristic zero which was proved by Lyubeznik.
\end{abstract}
 \maketitle

\section{Introduction} \label{sect1}

Let $R$ be a commutative Noetherian ring and $I \subset R$ be an
ideal. The $i$th local cohomology module of $R$ with support in $I$
is denoted by $H^i_I(R)$. In the present note, our main result
(Theorem \ref{Lyubeznik}) provides a different point of view of
$\Ann_R H^i_I(R)=0$ credited to Lyubeznik \cite[Corollary
3.6]{Lyu2}, where $R$ is a regular local ring containing a field of
characteristic $0$. Our way to prove the results is to use the
so-called $D$-modules. This method has played a decisive role in
many subsequent studies in the rings of characteristic zero.

Let $R$ be a commutative algebra over a field $k$ of characteristic
zero. We denote by $\End_k (R)$ the $k$-linear endomorphism ring.
The ring of $k$-linear differential operators $D_{R|k} \subseteq
\End_k (R)$ generated by the $k$-linear derivations $R \rightarrow
R$ and the multiplications by elements of $R$. By a $D_{R|k}$-module
we always mean a left $D_{R|k}$-module. The injective ring
homomorphism $R \rightarrow D_{R|k}$ that sends $r$ to the map $R
\rightarrow R$ which is the multiplication by $r$, gives $D_{R|k}$ a
structure of $R$-algebra. Every $D_{R|k}$-module $M$ is
automatically an $R$-module via this map. The natural action of
$D_{R|k}$ on $R$ makes $R$ a $D_{R|k}$-module. If $R=k[[x_1,
\ldots,x_n]]$ is a formal power series ring of $n$ variables $x_1,
\ldots, x_n$ over $k$, then $D_{R|k}$ is left and right Noetherian.
Moreover, $D_{R|k}$ is a simple ring. Noteworthy, the local
cohomology module $H^{i}_{I}(R),\ i \in \mathbb{Z}$ is a finitely
generated $D_{R|k}$-module. For a more advanced exposition based on
differential operators and undefined concepts the interested reader
might consult \cite{Bj}. For brevity we often write $D_R$ for
$D_{R|k}$ when there is no ambiguity about the field $k$.

\section{results}

\begin{lem} \label{A} Let $R$ be  $k[|x_1, \ldots,x_n|]$, a formal power series ring of $n$
variables $x_1, \ldots, x_n$ over a field $k$ of characteristic
zero. Suppose that $H^{i}_{I}(R)\neq 0$, then $\Ann_{D_R}
H^{i}_{I}(R)=0$.
\end{lem}

\begin{proof} Consider the homomorphism
\begin{equation}\label{hom}
   D_R \stackrel{f}{\longrightarrow} \Hom_{k}(H^{i}_{I}(R),H^{i}_{I}(R))
\end{equation}
of $D_R$-modules defined by $f(P)(m) = Pm,\ P \in D_R$ and $m \in
H^{i}_{I}(R)$ for all $i \in \mathbb{Z}$. The homomorphism $f$ is
injective, as $D_R$ is a simple ring, i.e. $\Ann_{D_R}
H^{i}_{I}(R)=\ker f=0$.
\end{proof}

\begin{lem} \label{B} Let $R$ be as in Lemma \ref{A}. Let $M$ be both an $R$ module and a
$D_R$-module. Then $\Ann_{D_R}M =0$ implies $\Ann_R M =0$.
\end{lem}

\begin{proof} Let $r \in \Ann_R M$ be an arbitrary element. As the endomorphism $\varphi_r :R \rightarrow
R$ with $\varphi_r(s) =rs$ for all $s \in R$ is an element of $D_R$
so from $rsM =0$ (for all $s \in R$) we have $\varphi_r(s)M =0$.
That is $\varphi_r(s) \in \Ann_{D_R}M =0$, i.e. $rs = \varphi_r(s)
=0$, for all $s \in R$. Hence, we have $r = 0$, as desired.
\end{proof}

\begin{thrm} \label{Lyubeznik}
Let $(R,\fm)$ be a regular local ring containing a field of
characteristic zero. Suppose that $H^i_I(R)\neq 0$. Then $\Ann_R
H^i_I(R)=0$.
\end{thrm}

\begin{proof}

Suppose $k = R/\fm$, where $\fm$ is the maximal ideal of $R$ and
$char(k) =0$. By virtue of \cite[Chapter IX, Appendice 2.]{Bo} there
exists a faithfully flat homomorphism from $(R,\fm)$ to a regular
local ring $(S,\fn)$ such that $S/\fn$ is the algebraic closure of
$k$. As $S$ is faithfully flat over $R$ then the homomorphism is
injective so $S$ contains a field. Moreover, it is known that
$\Ann_R H^i_I (R)= (\Ann_S H^i_{IS} (S)) \cap R$ and $H^i_I
(R)\otimes_R S \cong H^i_{IS}(S) \neq 0$, because of faithfully
flatness of $S$. Then, we may assume that $k$ is algebraically
closed. As $\hat{R}$ is also faithfully flat $R$-module, we may
assume that $R$ is complete, so that $R=k[[x_1,\ldots,x_n]]$ by the
Cohen Structure Theorem, where $n = \dim R$. Since,
$k[[x_1,\ldots,x_n]]$ has a $D_R$-module structure so, we are done
by Lemma \ref{B} and Lemma \ref{A}.
\end{proof}

\proof[Acknowledgements]

The author is grateful to Gennady Lyubeznik for pointing out an
error in a previous version of the paper. Many thanks to Josep
\`{A}lvarez-Montaner, who taught me the theory of D-modules,
generously.

\end{document}